\documentclass[12pt,twoside]{amsart}
\usepackage{amsfonts}
\usepackage{amsmath, amsthm, amssymb, mathrsfs, paralist,tabularx,supertabular, verbatim, amsthm, manfnt}
\usepackage[all]{xy}

\newcommand{\Z}{{\mathbb Z}}
\newcommand{\Q}{{\mathbb Q}}

\evensidemargin \oddsidemargin
\parskip=\medskipamount

\newtheorem{thm}{Theorem}[section]
\newtheorem{theorem}[thm]{Theorem}

\newtheorem{lemma}[thm]{Lemma}	
\newtheorem{proposition}[thm]{Proposition}

\theoremstyle{definition}

\theoremstyle{remark}

\newtheorem*{lemma*}{Lemma}

\numberwithin{equation}{section}

\title{Heights of Divisors of $x^n-1$}

\author{Lola Thompson}
\address{Department of Mathematics\\ 
6188 Kemeny Hall\\
Dartmouth College\\
Hanover, NH 03755, USA}
\email[] {Lola.Thompson@Dartmouth.edu}

\begin{document}

\begin{abstract} 
\noindent The height of a polynomial with integer coefficients is the largest coefficient in absolute value. Many papers have been written on the subject of bounding heights of cyclotomic polynomials. One result, due to H. Maier, gives a best possible upper bound of $n^{\psi(n)}$ for almost all $n$, where $\psi(n)$ is any function that approaches infinity as $n \rightarrow \infty$. We will discuss the related problem of bounding the maximal height over all polynomial divisors of $x^n - 1$ and give an analogue of Maier's result in this scenario.  
\end{abstract}

\maketitle 

\begin{center}
\textit{For my adviser Carl Pomerance on his $65^{th}$ birthday}
\end{center} 

\section{Introduction and statement of the principal result}

Let $\Phi_n(x)$ denote the $n^{th}$ cyclotomic polynomial. The $n^{th}$ cyclotomic polynomial is the unique monic irreducible polynomial over $\Q$ with the primitive $n^{th}$ roots of unity as its roots. It has integer coefficients. The degree of $\Phi_n(x)$ is $\varphi(n)$, where $\varphi$ is the Euler totient function. 

We define the \textit{height} of a polynomial with integer coefficients to be the largest coefficient in absolute value. We will denote the height of a polynomial $f$ by $H(f)$.

Much has been studied about $H(\Phi_n)$, which shall henceforth be denoted $A(n)$. In 1946, P. Erd\H{o}s stated that $\log A(n) \leq n^{(1 + o(1)) \log 2/ \log \log n}$. He held back its proof because of how complicated it was. R. C. Vaughan showed in 1975 that this inequality can be reversed for infinitely many $n$.

In 1949, P.T. Bateman gave a simple argument that if $k$ is a given positive integer then $A(n) \leq n^{2^{k-1}}$ if $n$ has exactly $k$ distinct prime factors. Let $\omega(n)$ denote the number of distinct prime factors of $n$. By taking the log of both sides of Bateman's inequality and using the fact that the maximal order of $\omega(n)$ is $\frac{\log n}{\log \log n}$ \cite[p.355]{hardywright}, one can show that Bateman's result implies Erd\H{o}s' result. Bateman's upper bound was improved upon by Bateman, C. Pomerance and Vaughan \cite{bpv84} in 1981, who showed that $A(n) \leq n^{2^{k-1}/k-1}$. They also showed that $A(n) \geq n^{2^{k-1}/k-1}/(5 \log n)^{2^{k-1}}$ holds for infinitely many $n$ with exactly $k$ distinct odd prime factors.

Related to these problems are questions concerning the maximal height over all divisors of $x^n-1$. It is well-known that $x^n-1 = \displaystyle\prod_{d \mid n} \Phi_d(x)$. Thus, $x^n-1$ has $\tau(n)$ distinct irreducible divisors, where $\tau(n)$ is the number of divisors of $n$. Therefore, $x^n-1$ has $2^{\tau(n)}$ divisors in $\Z[x]$. 

Let $B(n) =$ max$\{H(f) : f(x) \mid x^n-1, f(x) \in \Z[x]\}$. In particular, $A(n) \leq B(n)$ since $\Phi_n(x)$ divides $x^n-1$ and $B(n)$ is the maximum height over all divisors of $x^n-1$. In general, much less is known about $B(n)$ than $A(n)$. In 2005, Pomerance and N. Ryan \cite{pr07} proved that as $n \rightarrow \infty$, $\log B(n) \leq n^{(\log 3 + o(1))/\log \log n}$. They also showed that this inequality can be reversed for infinitely many $n$.  

In \cite{maier96}, H. Maier found an upper bound for $A(n)$ that holds for most $n$. 

\begin{theorem}[Maier] Let $\psi(n)$ be a function defined for all positive integers such that $\psi(n) \rightarrow \infty$ as $n \rightarrow \infty$. Then $A(n) \leq n^{\psi(n)}$ for almost all $n$, i.e., for all $n$ except for a set with asymptotic density $0$.\end{theorem}

Maier's upper bound has been shown to be best possible \cite{maier93}. In this paper, we consider an upper bound for $B(n)$ that holds for most $n$.

\begin{theorem} Let $\psi(n)$ be a function defined for all positive integers such that $\psi(n) \rightarrow \infty$ as $n \rightarrow \infty$. Then $B(n) \leq n^{\tau(n) \psi(n)}$ for almost all $n$, i.e., for all $n$ except for a set with asymptotic density $0$.\end{theorem}

It is not yet known whether this upper bound for $B(n)$ is best possible.

\section{Proof strategy for Theorem 1.2}

Since $x^n - 1 = \prod_{d \mid n} \Phi_d(x)$, then $B(n) = H \left(\prod_{d \in \mathcal{D}} \Phi_d(x) \right)$, where $\mathcal{D}$ is a subset of divisors of $n$ for which $\prod_{d \in \mathcal{D}} \Phi_d(x)$ has maximal height over all products of distinct cyclotomic polynomials dividing $x^n - 1$.

In \cite{pr07}, Pomerance and Ryan show that if $f_1, ..., f_k \in \Z[x]$ with deg $f_1 \leq \cdots \leq$ deg $f_k$ then\\ 
$H(f_1 ... f_k) \leq \prod_{i=1}^{k-1} (1 + $deg$ f_i) \prod_{i=1}^{k} H(f_i)$. Thus, when $n > 1$,
\begin{equation} \label{key inequality} B(n) = H\left(\prod_{d \in \mathcal{D}} \Phi_d(x) \right) \leq \prod_{d \in \mathcal{D}} (1 + \varphi(d)) \prod_{d \in \mathcal{D}} A(d) \leq n^{\# \mathcal{D}} \prod_{d \in \mathcal{D}} A(d) \leq n^{\tau(n)} \prod_{d \mid n} A(d). \end{equation}

Let $A_0(n) := \displaystyle\max_{d \mid n} A(d)$. Then from (2.1), $B(n) \leq n^{\tau(n)} A_0(n)^{\tau(n)}$, since $A(d) \leq A_0(n)$ for each $d \mid n$. So, if we show that $A_0(n) \leq n^{\psi(n)}$ for almost all $n$, we will have 
\begin{equation} \label{second key inequality} B(n) \leq n^{\tau(n)}A_0(n)^{\tau(n)} \leq n^{\tau(n)} \cdot n^{\tau(n) \psi(n)} = n^{\tau(n)(1 + \psi(n))}\end{equation} for almost all $n$. Since $\psi(n)$ is any function that goes to infinity as $n$ approaches infinity, we will have proved the theorem. 

Thus, we have reduced the proof of Theorem 1.2 to the following proposition, which shall be proven in section 4.\\

\begin{proposition} \label{main prop} We have $A_0(n) \leq n^{\psi(n)}$ for almost all $n$.\end{proposition}

\section{Key Lemmas}





Let $\omega(n)$ be defined as in section 1. Write the prime factorization of $n$ as $p_1^{e_1} \cdots p_{\omega(n)}^{e_{\omega(n)}}$, where $p_1 > p_2 > \cdots > p_{\omega(n)}$, $e_k \geq 1$ for $1 \leq k \leq \omega(n)$. Thus, we have functions $p_k = p_k(n)$ defined when $k \leq \omega(n)$. If $k > \omega(n),$ we let $p_k(n) = 1$.

To prove our proposition, we will show that for most integers, the size of the prime factors $p_k$ decreases rapidly on a logarithmic scale as $k$ increases.

\begin{lemma} \label{small lemma} Let $2 < \gamma < e$. The set $\{n : \omega(n) \geq \frac{\log \log n}{ \log \gamma} \}$ has density $0$. \end{lemma}

\begin{proof} Since $2 < \gamma < e$ then log $\gamma \in (0,1)$, so $1 < \frac{1}{\log \gamma}$. Now, the normal order of $\omega(n)$ is $\log \log n$ \cite[p.111]{pollack09}, so for each $\varepsilon > 0$, $\omega(n) < (1 + \varepsilon) \log \log n$ must hold, except for a set of $n$ with asymptotic density 0. In particular, since $\varepsilon = \frac{1}{\log \gamma} - 1 > 0$, then $\omega(n) < \frac{1}{\log \gamma} \log \log n$ for almost all $n$. \end{proof}

Let $\mu(n)$ be the M\"{o}bius function. From \cite[Lemma 5]{maier96}, we know that if $2 < \gamma < e$ then there is a constant $c(\gamma) > 0$ such that for each natural number $k < $log log $ x/$log $ \gamma,$
\[\#\{n \leq x : \mu(n) \neq 0, \ \log p_k > \gamma^{-k} \log x\} \ll xe^{-c(\gamma) k}.\]


The following lemma says that we can remove the restriction that $\mu(n) \neq 0$, i.e., we do not need to assume that $n$ is square-free.\\

\begin{lemma} \label{main lemma} Let $2 < \gamma < e$. Let $x > 1$. There are positive constants $c_0(\gamma), C_2$ such that for each natural number $k < \log \log x/\log \gamma$,
\[\#\{n \leq x : \log p_k > \gamma^{-k} \log x\} \leq C_2 xe^{-c_0(\gamma) k}.\] \end{lemma}

\begin{proof} We adopt the same strategy as in \cite{maier96}. The following is a classical result, due to Halberstam and Richert \cite[Thm 01]{ht88}: Let $f$ be a non-negative multiplicative function such that for some numbers $A$ and $B$ and for all numbers $y \geq 0$, we have
\begin{equation} 
\label{halberstam} \displaystyle\sum_{p \leq y} f(p) \log p \leq Ay, \ \ \ \ \ \ \ \ \displaystyle\sum_p \displaystyle\sum_{\nu \geq 2} \frac{f(p^\nu)}{p^{\nu}} \log p^{\nu} \leq B, \end{equation} where $p$ runs over primes and $\nu$ runs over integers. Then, for all numbers $x > 1$, \begin{equation} \label{halberstam2} \displaystyle\sum_{n \leq x} f(n) \leq (A + B + 1) \frac{x}{\log x} \displaystyle\sum_{n \leq x} \frac{f(n)}{n}. 
\end{equation}

We apply this theorem with $f(n) = b^{\omega([t,x],n)}$, where $w([t,x],n)$ is the number of distinct prime factors of $n$ in the interval $[t,x]$, with $t = x^{\gamma^{-k}}$, $b>1$ ($b$ will be specified later). In order to apply the theorem, we need to check that both conditions in \eqref{halberstam} are satisfied.

As usual, let $\theta(y) = \sum_{p \leq y} \log p$. Since $\theta(y) \leq 2y \log 2 < 2y$ \cite[p.108]{pollack09} then $$\displaystyle\sum_{p \leq y} f(p) \log p \leq 2by$$ for all $y$. Thus, the first condition is satisfied, with $A = 2b$.

Next, we show that the second condition is satisfied for a suitable number $B$, namely that the double sum converges. Consider the sum $\sum_p \sum_\nu \frac{\log p^{\nu}}{p^\nu} b^{\omega([t,y], p^\nu)}$, where $p$ runs over primes, $\nu \geq 2$. Since $\omega$ counts only distinct prime factors, we have $\omega([t,y], p^\nu) \leq 1$. So,
$$\sum_p \sum_{\nu \geq 2} \frac{\log p^{\nu}}{p^\nu} b^{\omega([t,y], p^\nu)} \leq b \sum_p \left(\frac{2 \log p}{p^2} + \frac{3 \log p}{p^3} + \cdots \right) = b \sum_p \left(\frac{2}{p^2} + \frac{3}{p^3} + \cdots \right) \log p.$$ 
It is easy to see that \begin{equation} \label{midstep} \sum_p \left(\frac{2}{p^2} + \frac{3}{p^3} + \cdots \right) \log p = 2 \sum_p \frac{\log p}{p(p-1)}\end{equation} holds, and that the sum in \eqref{midstep} is less than $4$. Thus, the second condition is satisfied, with $B = 4b$.

Therefore, by \eqref{halberstam2}, we have 
\begin{equation} 
\label{halberstam result} \displaystyle\sum_{n \leq x} b^{\omega([t,x], n)} \leq (2b + 4b + 1) \frac{x}{\log x} \displaystyle\sum_{n \leq x} \frac{f(n)}{n} \leq 7b \frac{x}{\log x} \displaystyle\sum_{n \leq x} \frac{f(n)}{n}. \end{equation}

Now, $\sum_{n \leq x} \frac{f(n)}{n} \leq \prod_{p \leq x} \left(1 + \frac{f(p)}{p} + \frac{f(p^2)}{p^2} + \cdots \right)$, since $f$ is a non-negative multiplicative function (certainly all prime factors of each $n \leq x$ are in this product). Taking the $\log$ of both sides, 
we have \begin{eqnarray*} \log \left(\displaystyle\sum_{n \leq x} \frac{f(n)}{n} \right) & \leq & \log \displaystyle\prod_{p \leq x} \left(1 + \frac{f(p)}{p} + \frac{f(p^2)}{p^2} + \cdots \right) \\ & = & \log \displaystyle\prod_{p \leq x} \left(1 + f(p) \left(\frac{1}{p} + \frac{1}{p^2} + \cdots \right) \right) \\ & = & \log \displaystyle\prod_{p \leq x} \left(1 + \frac{f(p)}{p-1} \right) = \displaystyle\sum_{p \leq x} \log \left(1 + \frac{f(p)}{p-1} \right). \end{eqnarray*} 
Thus, $$ \log \left(\displaystyle\sum_{n \leq x} \frac{f(n)}{n} \right) \leq \displaystyle\sum_{p \leq x} \frac{f(p)}{p-1} = \displaystyle\sum_{p < t} \frac{1}{p-1} + \displaystyle\sum_{t \leq p \leq x} \frac{b}{p-1},$$
since $f(p) = 1$ when $p < t$ and $f(p) = b$ when $t \leq p \leq x$. By Mertens' first theorem \cite[p.92]{pollack09},
$$ \displaystyle\sum_{p < t} \frac{1}{p-1} + \displaystyle\sum_{t \leq p \leq x} \frac{b}{p-1} = \log \log x + (b-1)(\log \log x - \log \log t) + O(b).$$ Let $\alpha$ be the constant associated with $O(b)$. After undoing the logarithms, we are left with 
\begin{equation}
\label{no logs simplification} \displaystyle\sum_{n \leq x} \frac{f(n)}{n} \leq C_1 \log x\left(\frac{\log x}{\log t} \right)^{b-1}, 
\end{equation} where $C_1 = e^{\alpha b}$. 
Inserting \eqref{no logs simplification} into \eqref{halberstam result}, we have 
\begin{equation} 
\label{new halberstam result} \displaystyle\sum_{n \leq x} b^{\omega([t,x], n)} \leq 7bC_1 x \left(\frac{\log x}{\log t} \right)^{b-1}. 
\end{equation}

Let $C_2 = 7bC_1$. Let $$N = \#\{n \leq x : \omega([t,x], n) > \frac{(1 + \varepsilon)(b-1)}{\log b} (\log \log x - \log \log t)\}.$$ Using \eqref{new halberstam result}, we have $$N b^{\frac{(1 + \varepsilon)(b-1)}{\log b}(\log \log x - \log \log t)} \leq \displaystyle\sum_{n \leq x} b^{\omega([t,x], n)} \leq C_2 x \left(\frac{\log x}{\log t} \right)^{b-1}.$$ But $$b^{\frac{(1+\varepsilon)(b-1)}{\log b}(\log \log x - \log \log t)} = e^{(1 + \varepsilon)(b-1)(\log \log x - \log \log t)} = \left(\frac{\log x}{\log t} \right)^{(1 + \varepsilon)(b - 1)}.$$ So 
$$N \leq \frac{C_2 x (\frac{\log x}{\log t})^{b-1}}{(\frac{\log x}{\log t})^{(1 + \varepsilon)(b-1)}} = C_2 x\left(\frac{\log x}{\log t}\right)^{-\varepsilon(b - 1)}.$$ In other words, 
\begin{equation} 
\label{omega bound} \omega([t,x], n) \leq \frac{(1 + \varepsilon)(b - 1)}{\log b} (\log \log x - \log \log t) \end{equation} for all $n \leq x$ except for a set of cardinality at most $C_2 x (\frac{\log x}{\log t})^{- \varepsilon (b - 1)}$. 

Now, fix $\varepsilon > 0$, $b > 1$ such that $\frac{(1+\varepsilon)(b-1)}{\log b} \log \gamma \leq 1$. Let $k < \log \log x/\log \gamma$. Recall that $t = x^{\gamma^{-k}}$. Then, if $\log p_k > \gamma^{-k} \log x$, we have 
\begin{equation} 
\label{last one} \omega([t,x], n) \geq k \geq \frac{(1+\varepsilon)(b-1)}{\log b} k \log \gamma. \end{equation} 
Since $k \log \gamma = \log \log x  - \log \log t,$ we have $\omega([t,x], n) \geq \frac{(1 + \varepsilon)(b - 1)}{\log b} (\log \log x - \log \log t)$. But this contradicts \eqref{omega bound} except for a set of cardinality at most $C_2 x(\frac{\log x}{\log t})^{-\varepsilon(b - 1)}$. Thus, the set of $n \leq x$ with $\log p_k > \gamma^{-k} \log x$ has a cardinality of at most $C_2 x (\frac{\log x}{\log t})^{-\varepsilon(b-1)}.$ Since $t = x^{\gamma^{-k}}$, we have $$\# \{n \leq x : \log p_k > \gamma^{-k} \log x\} \leq C_2 x e^{-k \varepsilon (b-1) \log(\gamma)}.$$ Taking $c_0(\gamma) = \varepsilon (b-1) \log(\gamma)$, we obtain the desired result.\end{proof}

The following lemma says that, except for a sparse set of integers $n$, $\log p_k$ is small when $k$ is sufficiently large.

\begin{lemma} \label{corollary} Let $2 < \gamma < e$. Let $\varepsilon > 0$ be arbitrary and let $k_0 = \frac{\log(\varepsilon(1 - e^{-c_0(\gamma)})/C_2)}{-c_0(\gamma)}$, where $c_0(\gamma)$ and $C_2$ are as in Lemma \ref{main lemma}. Then, for $x$ sufficiently large, the set $\{n \leq x : \log p_k > \gamma^{-k} \log x$ for some $ k \geq k_0 \}$ has cardinality at most $2 \varepsilon x$. \end{lemma}

\begin{proof} Fix $\varepsilon > 0$. Let $\mathcal{S} = \{n \leq x : \log p_k > \gamma^{-k} \log x \ \mathrm{for \ some} \ k \geq k_0\}$ and let $S_k =  \{n \leq x : \log p_k > \gamma^{-k} \log x \}$. By Lemma \ref{main lemma}, we have $$\# \ \mathcal{S} \leq \displaystyle\sum_{k = \lceil k_0 \rceil}^{\lfloor \frac{\log \log x}{\log \gamma} \rfloor} \# \ \mathcal{S}_k + \# \{n : \omega(n) > \frac{\log \log x}{\log \gamma} \} \leq \displaystyle\sum_{k = \lceil k_0 \rceil}^\infty C_2 x e^{-c_0(\gamma) k} + \varepsilon x$$ for sufficiently large $x$, since $\{n : \omega(n) > \frac{\log \log x}{\log \gamma} \}$ has density 0 by Lemma \ref{small lemma}. But the sum on the right is a convergent geometric series, so $$\# \ \mathcal{S} \leq \frac{C_2 x e^{-c_0(\gamma)k_0}}{1 - e^{-c_0(\gamma)}} + \varepsilon x.$$ Thus, using the definition of $k_0$, \[\#\{n \leq x : \log  p_k > \gamma^{-k}  \log  x \ \mathrm{for \ some} \ k \geq k_0 \} \leq 2 \varepsilon x.\] \end{proof}

\section{Proof of proposition 2.1}

\begin{proof} Maier shows in [2] that if $\psi(n)$ is any function defined on all positive integers $n$ such that $\psi(n) \rightarrow \infty$ as $n \rightarrow \infty$ then $A(n) \leq n^{\psi(n)}$ for almost all $n$. Key to this proof is the fact that 
\begin{equation} 
\label{maierkey} \log A(n) \leq C \displaystyle\sum_{k=1}^{\omega(n)} 2^k \log p_k 
\end{equation}
for all square-free integers $n$, where $C > 0$ is a constant and $p_k = p_k(n)$ is as above. 

We define the radical of $n$, denoted $\mathrm{rad}(n)$, to be the largest square-free divisor of $n$. Since $\Phi_n(x) = \Phi_{\mathrm{rad}(n)}(x^{n/\mathrm{rad}(n)})$, the coefficients of $\Phi_n(x)$ are the same as the coefficients of $\Phi_{\mathrm{rad}(n)}(x)$. Thus, $A(n) = A(\mathrm{rad}(n))$. As a result, we can use \eqref{maierkey} for any positive integer $n$, since $$\log A(n) = \log A(\mathrm{rad}(n)) \leq C\sum_{k=1}^{\omega(n)} 2^k \log p_k.$$ 

For each $d$ dividing $n$, let $d = p_{1,d}^{e_{1,d}} p_{2,d}^{e_{2,d}} \cdots p_{\omega(d), d}^{e_{\omega(d), d}}$, where $p_{1,d} > p_{2,d} > \cdots > p_{\omega(d), d}$ and $e_{k,d} \geq 1$ for $1 \leq k \leq \omega(d).$ Also, let $p_{k,d} = 1$ for $k > \omega(d)$. Since $d \mid n$ then the primes dividing $d$ also divide $n$. Thus, $p_{k, d} \leq p_k$ for all $k$, so $$\displaystyle\sum_{k=1}^{\omega(d)} 2^k \log p_{k,d} \leq \displaystyle\sum_{k=1}^{\omega(d)} 2^k \log p_k \leq \displaystyle\sum_{k = 1}^{\omega(n)} 2^k \log p_k.$$ 
Thus, log $A(d) \leq C \displaystyle\sum_{k=1}^{\omega(n)} 2^k $ log $ p_k$ holds for all $n$ and for all $d \mid n$. Since log $A_0(n) =  \log A(d) $ for some $d \mid n$ we then have $\log A_0(n) \leq C \displaystyle\sum_{k=1}^{\omega(n)} 2^k \log p_k.$

Let $\varepsilon > 0$ be arbitrary and let $k_0$ be as in Lemma \ref{corollary}. Combining the above inequality with Lemma \ref{corollary}, we have 
\begin{align} \log A_0(n) \leq C \displaystyle\sum_{k=1}^{\omega(n)} 2^k \log p_k = C \displaystyle\sum_{k \leq \lfloor k_0 \rfloor} 2^k \log p_k \ + C \displaystyle\sum_{k = \lfloor k_0 \rfloor +1}^{\omega(n)} 2^k \log p_k \\ \label{*} \leq C\displaystyle\sum_{k \leq \lfloor k_0 \rfloor} 2^k \log p_k  + C \displaystyle\sum_{k= \lfloor k_0 \rfloor +1}^{\omega(n)} (2/\gamma)^k \log x \end{align} for all $n \leq x$ except for a set with cardinality $\leq 2 \varepsilon x$. Since $2 < \gamma < e$ then $(2/\gamma) <1$. Hence, $\sum_{k= \lfloor k_0 \rfloor}^{\omega(n)} (2/\gamma)^k $ is part of a convergent geometric series, so it is bounded above by some positive constant $L$ that is independent of $n$. 

Now, if $\sqrt{x} \leq n \leq x$ then $2 \log n > \log x$, so 
$$\displaystyle\sum_{k = \lfloor k_0 \rfloor + 1}^{\omega(n)} (2/\gamma)^k \log x \leq 2 \log n \displaystyle\sum_{k = \lfloor k_0 \rfloor + 1}^{\omega(n)} (2/\gamma)^k = 2L \log n. $$ Then, if $n$ is such that \eqref{*} holds, \begin{align*} \log A_0(n) \leq C \displaystyle\sum_{k \leq \lfloor k_0 \rfloor} 2^k \log p_k \ + 2L \log n \\ \leq 2^{\lfloor k_0 \rfloor} C \displaystyle\sum_{k \leq \lfloor k_0 \rfloor} \log p_k \ + 2L  \log n\\ = 2^{\lfloor k_0 \rfloor} C\log (\displaystyle\prod_{k \leq \lfloor k_0 \rfloor} p_k) + 2L \log n \\ \leq \log (n^{2^{\lfloor k_0 \rfloor} C}) + \log (n^{2L}). \end{align*} Thus, $A_0(n) \leq n^{2^{\lfloor k_0 \rfloor}C} \cdot n^{2L}$. 
Then, we have $$A_0(n) \leq n^{2^\frac{\log (\varepsilon(1 - e^{-c_0(\gamma)})/C_2)}{-c_0(\gamma)}} \cdot n^{2L} \leq n^{e^\frac{\log (\varepsilon(1-e^{-c_0(\gamma)})/C_2)}{-c_0(\gamma)}} \cdot n^{2L} = n^{(\varepsilon(1-e^{-c_0(\gamma)})/C_2)^{-c_0(\gamma)}} \cdot n^{2L}.$$ As mentioned, this holds for all $n$ with $\sqrt{x} \leq n \leq x$ and for which \eqref{*} holds. Therefore, for any $\varepsilon > 0$ there is a constant $C_3 = (\frac{\varepsilon(1-e^{-c_0(\gamma)})}{C_2})^{-c_0(\gamma)} + 2L$ such that for all sufficiently large $x$, every $n \leq x$ satisfies $A_0(n) \leq n^{C_3}$, except for at most $2 \varepsilon x + \sqrt{x}$ of them. Since $\varepsilon > 0$ is arbitrary, this proves Proposition \ref{main prop}, which concludes the proof of our main theorem. \end{proof}

\textit{Acknowledgements.} I would like to thank my adviser, Carl Pomerance, for the careful guidance and encouragement that he provided while I was writing this paper. 

\providecommand{\bysame}{\leavevmode\hbox
to3em{\hrulefill}\thinspace}
\providecommand{\MR}{\relax\ifhmode\unskip\space\fi MR }
\providecommand{\MRhref}[2]{%
  \href{http://www.ams.org/mathscinet-getitem?mr=#1}{#2}
} \providecommand{\href}[2]{#2}

\end{document}